\documentclass[11pt, leqno]{amsart}
\usepackage{amsmath, amssymb, latexsym}
\usepackage{mathrsfs}
\usepackage{enumerate}
\usepackage{color}
\usepackage[colorlinks=true, pdfstartview=FitV, linkcolor=blue,%
citecolor=blue, urlcolor=blue]{hyperref}
\usepackage[all, knot]{xy}
\xyoption{arc}

\numberwithin{equation}{section}

\newtheorem{theorem}{Theorem}[section]
\newtheorem{corollary}[theorem]{Corollary}
\newtheorem{lemma}[theorem]{Lemma}
\newtheorem{proposition}[theorem]{Proposition}

\theoremstyle{definition}
\newtheorem{definition}[theorem]{Definition}

\newcommand{\Q}{\mathbb{Q}}

\newcommand{\Z}{\mathbb{Z}}
\newcommand{\N}{\mathbb{N}}
\newcommand{\h}{\mathfrak{h}}
\newcommand{\g}{\mathfrak{g}}

\title[Classical limit of quantum Borcherds-Bozec algebras]
{Classical limit of Quantum Borcherds-Bozec
Algebras}

\author[Zhaobing Fan]{Zhaobing Fan}
\address{Harbin Engineering University,
Harbin, China}
\email{fanz@ksu.edu}
\thanks{ }

\author[Seok-Jin Kang]{Seok-Jin Kang}
\address{Korea Research Institute of Arts and Mathematics,
Asan-si, Chungcheongnam-do, 31551, Korea}
\email{soccerkang@hotmail.com}

\thanks{}

\author[Young Rock Kim]{Young Rock Kim${}^{*}$}
\address{Graduate School of Education, Hankuk University of Foreign Studies, Seoul, 02450,  Korea}
\email{rocky777@hufs.ac.kr} %
\thanks{${}^{*}$ Corresponding author. All authors have equal contributions.}

\author[Bolun Tong]{Bolun Tong}
\address{Harbin Engineering University,
Harbin, China}
\email{tbl\_2019@hrbeu.edu.cn}

\address{}


\begin{document}

\begin{abstract}

Let $\g$ be a Borcherds-Bozec algebra, $U(\g)$ be its universal enveloping algebra and
$U_{q}(\g)$ be the corresponding quantum Borcherds-Bozec algebra. We show that the
classical limit of $U_{q}(\g)$ is isomorphic to $U(\g)$ as Hopf algebras. Thus $U_{q}(\g)$ can be regarded
as a quantum deformation of $U(\g)$. We also give explicit formulas for the commutation
relations among the generators of $U_{q}(\g)$.

\end{abstract}

\maketitle

\section*{Introduction}
\vskip 2mm

The {\it quantum Borcherds-Bozec algebras} were introduced by T. Bozec in his research of perverse sheaves theory for quivers with loops \cite{Bozec2014b, Bozec2014c,BSV2016}. They can be treated as a further generalization of quantum generalized Kac-Moody algebras. Even though they use the same Borcherds-Cartan data, the construction of the quantum groups are quite different.

\vskip 2mm

More precisely, the quantum Borcherds-Bozec algebras have more generators and defining relations than quantum generalized Kac-Moody algebras. For each simple root $\alpha_i$ with imaginary index, there are infinitely many generators $e_{il}, f_{il}$$(l \in \Z_{>0})$ whose degrees are $l$ multiples of $\alpha_i$ and $-\alpha_i$. Bozec deals with these generators by treating them as similar positions as divided powers $\theta_i^{(l)}$ in Lusztig algebras.

\vskip 2mm

 Bozec gave the general definition of Lusztig sheaves for arbitrary quivers (possibly with multiple loops) and constructed the canonical basis for the positive half of a quantum Borcherds-Bozec algebra in terms of simple perverse sheaves (cf. \cite{Lus90}). In \cite{Bozec2014c}, he studied the crystal basis theory for quantum Borcherds-Bozec algebras. He defined the notion of Kashiwara operators and abstract crystals, which provides an important framework for Kashiwara's grand-loop argument (cf. \cite{Kas91}). He also gave a geometric construction of the crystal for the negative half of a quantum Borcherds-Bozec algebra based on the theory of Lusztig perverse sheaves associated to quivers with loops (cf. \cite{KS1997, KKS2009}), and gave a geometric realization of generalized crystals for the integrable highest weight representations $via$ Nakajima's quiver varieties (cf. \cite{Saito2002, KKS2012}).

\vskip 2mm
For a Kac-Moody algebra $\g$,  G.  Lusztig showed that the integrable highest weight module $\overline{L}$ over $U(\g)$ can be
deformed to those integrable highest weight module $L$ over $U_q(\g)$ in such a way that the dimensions of weight spaces are invariant under the deformation (cf. \cite{Lus88, Jan}).
Let $\mathscr A=\Q[q,q^{-1}]$ be the Laurent polynomial rings, Lusztig constructed a $\mathscr A$-subalgebra $U_{\mathscr A}$ of $U_q(\g)$ generated by divided powers and $k_i^{\pm}$, and defined a $U_{\mathscr A}$-submodule $L_{\mathscr A}$ of $L$. He proved that $F_0\otimes_{\mathscr A}L_{\mathscr A}$ is isomorphic to $\overline L$ as $U(\g)$-modules, where $F_0=\mathscr A / I$ and $I$ is the ideal of $\mathscr A$ generated by $(q-1)$.
\vskip 2mm

In \cite[Chpter 3]{HK2002}, J. Hong and S.-J. Kang modified Lusztig's approach to show that the $U_q(\g)$ is a deformation of $U(\g)$ as a Hopf algebra and show that a highest weight $U(\g)$-module admits a deformation to a highest weight $U_q(\g)$-module. They used the $\mathbb A_1$-form of $U_q(\g)$ and highest weight $U_q(\g)$-module, where $\mathbb A_1$ is the localization of $\Q[q]$ at the ideal $(q-1)$. We can see that $\mathscr A=\Q[q,q^{-1}]\subseteq \mathbb A_1$.

\vskip 2mm
In this paper, we study the classical limit theory of quantum Borcherds-Bozec algebras. We first review some basic notions of Borcherds-Bozec algebras and quantum Borcherds-Bozec algebras. For their representation theory, the readers may refer to \cite{Kang2019, Kang2019b}. As we show in Appendix, the commutation relations between $e_{il}$ and $f_{jk}$ are rather complicated. For the aim of classical limit, we need another set of generators. Thanks to Bozec, there exists an alternative set of primitive generators in $U_q{(\g)}$,  which we denote by $s_{il}$ and $t_{il}$. They satisfy a simpler set of commutation relations $$s_{il}t_{jk}-t_{jk}s_{il}=\delta_{ij}\delta_{lk}\tau_{il}(K_i^l-K_i^{-l})$$
for some constants $\tau_{il}\in \Q(q)$.
Using Lusztig's approach, we prove that these generators also satisfy the Serre-type relations (cf. \cite[Chapter 1]{Lusztig}).
\vskip 2mm
In Section 3, we define the $\mathbb A_1$-form of quantum Borcherds-Bozec algebras and their highest weight representations.
 We show that the triangular decomposition of $U_q(\g)$ carries over to $\mathbb A_1$-form. In Section 4, we study the process of taking the limit $q\rightarrow1$. Let $U_1=\Q\otimes_{\mathbb A_1}U_{\mathbb A_1}$ be a $\Q$-algebra, where $U_{\mathbb A_1}$ is the $\mathbb A_1$-form of $U_q(\g)$. We prove that the classical limit $U_1$ of $U_q(\g)$ is isomorphic to the universal enveloping algebra $U(\g)$ as Hopf algebras, and when we take the classical limit, the Verma module and highest weight modules of $U_q(\g)$ {\it tend to} those Verma module and highest weight modules of $U(\g)$, respectively. Finally, we give the concrete commutation relations between the generators $e_{il}$ and $f_{jk}$ of $U_q(\g)$ in Appendix, they have an interesting combinatorial structure.
\vskip 2mm

{\it Acknowledgements}. Z. Fan is partially supported by the NSF of China grant 11671108 and the Fundamental Research Funds for the central universities GK2110260131. S.-J.  Kang was supported by Hankuk University of Foreign Studies Research Fund. Y. R. Kim was supported by the Basic Science Research Program of the NRF (Korea) under grant No. 2015R1D1A1A01059643.

\vskip 8mm

\vskip 10mm

\section{Borcherds-Bozec algebras}

\vskip 2mm

Let $I$ be an index set possibly countably infinite. An integer-valued matrix $A=(a_{ij})_{i,j\in I}$ is called an {\it even symmetrizable Borcherds-Cartan matrix} if it satisfies the following conditions:
\begin{itemize}
\item[(i)] $a_{ii}=2, 0, -2, -4, \ldots$,

\item[(ii)] $a_{ij} \in \Z_{\le 0}$ for $i \neq j$,

\item[(iii)] there is a diagonal matrix $D=\text{diag} (r_i \in
\Z_{>0} \mid i \in I)$ such that $DA$ is symmetric.
\end{itemize}

\vskip 3mm
Set $I^{\text{re}}:=\{ i \in I \mid a_{ii}=2 \}$, the set of {\it
real indices} and $I^{\text{im}}:= \{ i \in I \mid a_{ii} \le 0 \}$,
the set of {\it imaginary indices}. We denote by $I^{\text{iso}}: =
\{ i \in I \mid a_{ii}=0 \}$ the set of {\it isotropic indices}.

\vskip 3mm

A {\it Borcherds-Cartan datum} consists of

\begin{itemize}
\item[(a)] an even symmetrizable Borcherds-Cartan matrix
$A=(a_{ij})_{i,j \in I}$,

\item[(b)] a free abelian group $P^{\vee}=\left( \bigoplus_{i\in I}{\Z h_i}\right)\oplus \left(\bigoplus_{i\in I}{\Z d_i}\right)$, the {\it dual weight lattice},

\item[(c)] $\h=\Q \otimes_{\Z}P^{\vee}$, the {\it Cartan subalgebra},

\item[(d)]$P=\{\lambda \in \h^* \mid \lambda(P^{\vee})\subseteq\Z \}$, the {\it weight lattice},

\item[(e)] $\Pi^{\vee} = \{ h_i \in P^{\vee} \mid i \in I \}$, the set of {\it simple coroots},

\item[(f)]$\Pi = \{ \alpha_i \in P \mid i \in I \}$, the set of {\it simple roots},
which  is linearly independent over $\Q$ and satisfies
  \begin{equation*}
   \alpha_j(h_i)=a_{ij},\ \alpha_j(d_i)=\delta_{ij} \ \ \text{for all} \ i,j  \in I.
  \end{equation*}

\item[(g)]for each $i \in I$, there is an element $\Lambda_i \in P$     such that
  \begin{equation*}
   \Lambda_i(h_j)=\delta_{ij}, \ \Lambda_i(d_j)=0\ \ \text{for all} \ i,j  \in I.
  \end{equation*}
  The $\Lambda_i(i\in I)$ are called the {\it fundamental weights}.
\end{itemize}

\vskip 3mm

We denote by
$$P^+:=\{\lambda \in P \mid \lambda(h_i)\geq 0 \ \ \text{for all} \ i \in I \}$$
the set of {\it dominant integral weights}. The free abelian group $Q:=\bigoplus_{i \in I} {\Z \alpha_i}$ is called the {\it root lattice}. Set $Q_+=\sum_{i \in I}{\Z_{\geq0}\alpha_i}$ and $Q_-=-Q_+$. For $\beta=\sum k_i\alpha_i \in Q_+ $, we define its {\it hight} to be $ \text{ht} (\beta):=\sum k_i$.

\vskip 3mm

There is a non-degenerate symmetric bilinear form $( \ , \ )$ on $\h^*$ satisfying
$$(\alpha_i,\lambda)=r_i\lambda(h_i) \ \ \text{for all} \ \lambda \in \h^*,$$
and therefore we have $$(\alpha_i,\alpha_j)=r_ia_{ij}=r_ja_{ji} \ \ \text{for all} \  i,j \in I.$$

\vskip 3mm

For $i \in I^{\text{re}}$, we define the {\it simple reflection}
$\omega_{i} \in GL(\h^{*})$ by
$$\omega_{i} (\lambda) =  \lambda-\lambda(h_{i})
\alpha_{i} \ \ \text{for} \ \lambda \in \h^{*}.$$
The subgroup $W$
of $GL(\h^*)$ generated by $\omega_{i}$ $(i \in I^{\text{re}})$ is called
the {\it Weyl group} of $\g$. One can easily verify that the
symmetric bilinear form $(\ , \ )$ is $W$-invariant.

\vskip 3mm

Let $I^{\infty}:= (I^{\text{re}} \times \{1\}) \cup (I^{\text{im}}
\times \Z_{>0})$. For simplicity, we will often write $i$ for $(i,1)$ if $i \in I^{\text{re}}$.

\vskip 3mm

\begin{definition}
The {\it Borcherds-Bozec algebra} $\g$ associated with a
Borcherds-Cartan datum $(A, P, \Pi, P^{\vee}, \Pi^{\vee})$ is the
Lie algebra over $\Q$ generated by the elements $e_{il}$, $f_{il}$
$((i,l) \in I^{\infty})$ and $\h$ with defining relations
\begin{equation}\label{eq:bozec-alg}
\begin{aligned}
& [h, h']=0 \ \ \text{for} \ h, h' \in \h,\\
& [e_{ik}, f_{jl}] = k \, \delta_{ij} \, \delta_{kl} \, h_{i} \ \
\text{for} \ i,j \in I, k, l \in \Z_{>0},\\
& [h, e_{jl}]=l  \alpha_{j}(h)e_{jl}, \quad
 [h, f_{jl}]= -l \alpha_{j}(h)f_{jl}, \\
& (\text{ad} e_{i})^{1-la_{ij}}(e_{jl}) = 0 \ \ \text{for} \ i\in
I^{\text{re}}, \, i \neq (j,l), \\
& (\text{ad} f_{i})^{1-la_{ij}}(f_{jl}) = 0 \ \ \text{for} \ i\in
I^{\text{re}}, \, i \neq (j,l), \\
& [e_{ik}, e_{jl}] = [f_{ik}, f_{jl}] =0 \ \ \text{for} \ a_{ij}=0.
\end{aligned}
\end{equation}
\end{definition}

\vskip 3mm

Let $U(\g)$ be the universal enveloping algebra of $\g$. Since we have the following equations in $U(\g)$
$$(\text{ad}x)^m(y)=\sum_{k=0}^m(-1)^k
{\binom m k}x^{m-k}yx^k \ \  \text{for} \ x,y \in U(\g),m \in \Z_{\geq0},
$$
we obtain the presentation of $U(\g)$ with generators and relations given below.
\vskip 3mm

\begin{proposition}\label{uea}
{\rm The universal enveloping algebra $U(\g)$ of $\g$ is an
associative algebra over $\Q$ with unity generated by $e_{il}$, $f_{il}$
$((i,l) \in I^{\infty})$ and $\h$ subject to the following defining relations
\begin{equation}\label{eq:un enveloping alg}
\begin{aligned}
& hh'=h'h \ \ \text{for} \ h, h' \in \h,\\
& e_{ik}f_{jl}-f_{jl}e_{ik} = k \, \delta_{ij} \, \delta_{kl} \, h_{i} \ \ \text{for} \ i,j \in I, k, l \in \Z_{>0},\\
& he_{jl}-e_{jl}h=l  \alpha_{j}(h)e_{jl}, \quad
 hf_{jl}-f_{jl}h= -l \alpha_{j}(h)f_{jl}, \\
& \sum_{k=0}^{1-la_{ij}}(-1)^k
{\binom {1-la_{ij}} k}
{e_i}^{{1-la_{ij}}-k}e_{jl}\, e_i^k=0 \ \ \text{for} \ i\in
I^{\text{re}}, \, i \neq (j,l), \\
& \sum_{k=0}^{1-la_{ij}}(-1)^k
{\binom {1-la_{ij}} k}
{f_i}^{{1-la_{ij}}-k}f_{jl}\, f_i^k=0 \ \ \text{for} \ i\in
I^{\text{re}}, \, i \neq (j,l), \\
& e_{ik}e_{jl}-e_{jl}e_{ik} = f_{ik}f_{jl}-f_{jl}f_{ik} =0 \ \ \text{for} \ a_{ij}=0.
\end{aligned}
\end{equation} }
\end{proposition}

\vskip 3mm

The universal enveloping algebra $U(\g)$ has a Hopf algebra structure given by
\begin{equation}\label{str}
\begin{aligned}
&\Delta(x)=x\otimes 1+1\otimes x,\\
&\varepsilon(x)=0,\\
& S(x)=-x  \ \ \text{for} \ x \in \g,
\end{aligned}
\end{equation}
where $\Delta: U(\g) \rightarrow U(\g) \otimes U(\g)$ is the comultiplication,
$\varepsilon: U(\g) \rightarrow \Q$ is the counit, and $S:U(\g) \rightarrow U(\g)$ is the antipode.

\vskip 3mm

Furthermore, by the Poincar\'{e}-Brikhoff-Witt Theorem, the universal enveloping algebra also has the {\it triangular decomposition}
\begin{equation}\label{trian}
U(\g)\cong U^-(\g)\otimes U^0(\g) \otimes U^+(\g),
\end{equation}
where $U^+(\g)$ (resp. $U^0(\g)$ and $U^-(\g)$) be the subalgebra of $U(\g)$ generated by the elements $e_{il}$ (resp. $\h$ and $f_{il}$) for $(i,l) \in I^{\infty}$.

\vskip 3mm

In \cite{Kang2019}, Kang studied the representation theory of the Borcherds-Bozec algebras. We list some results that we will use later.

\vskip 3mm

\begin{proposition} \label{prop:hwmodule} \cite{Kang2019} \

\vskip 2mm

{\rm
\begin{enumerate}
\item[(a)] Let $\lambda \in P^{+}$ and $V(\lambda)=U(\g) v_{\lambda}$ be
the irreducible highest weight $\g$-module. Then we have
\begin{equation}\label{eq:hwmodule}
\begin{aligned}
& f_{i}^{ \lambda (h_{i}) + 1} v_{\lambda} =0 \ \
\text{for} \ i \in I^{\text{re}}, \\
& f_{il} \, v_{\lambda} =0 \ \ \text{for} \ (i,l) \in I^{\infty} \
\text{with} \ \lambda (h_{i}) =0.
\end{aligned}
\end{equation}

\vskip 2mm

\item[(b)] Every highest weight $\g$-module with highest weight $\lambda
\in P^{+}$ satisfying \eqref{eq:hwmodule} is isomorphic to $V(\lambda)$.
\end{enumerate}}

\end{proposition}

\section{quantum Borcherds-Bozec algebras}

\vskip 2mm

Let $q$ be an indeterminate and set
$$q_i=q^{r_i},\quad q_{(i)}=q^{\frac{(\alpha_i,\alpha_i)}{2}}.$$
Note that $q_i= q_{(i)}$ if $i \in I^{\text re}$. For each $i \in I^{\text re}$ and $n \in \Z_{\geq 0}$, we define
$$[n]_i=\frac{q_{i}^n-q_{i}^{-n}}{q_{i}-q_{i}^{-1}},\quad [n]_i!=\prod_{k=1}^n [k]_i,\quad
{\begin{bmatrix} n \\ k \end{bmatrix}}_i=\frac{[n]_i!}{[k]_i![n-k]_i!}$$

\vskip 3mm

Let $\mathscr F=\Q(q)\left< f_{il} \mid (i,l) \in I^{\infty} \right>$ be the free associative algebra over $\Q(q)$ generated by the symbols $f_{il}$ for $(i,l)\in I^{\infty}$. By setting $\text{deg} f_{il}= -l \alpha_{i} $, $\mathscr F$ become a $Q_-$-graded algebra. For a homogeneous element $u$ in $\mathscr F$, we denote by $|u|$ the degree of $u$, and for any $A \subseteq Q_{-}$, set
${\mathscr F}_{A}=\{ x\in {\mathscr F} \mid |x| \in A \}$.

\vskip 3mm

We define a {\it twisted} multiplication on $\mathscr F\otimes\mathscr F$ by
$$(x_1\otimes x_2)(y_1\otimes y_2)=q^{-(|x_2|,|y_1|)}x_1y_1\otimes x_2y_2,$$
and equip $\mathscr F$ with a co-multiplication $\delta$ defined by
$$\delta(f_{il})=\sum_{m+n=l}q_{(i)}^{-mn}f_{im}\otimes f_{in} \ \ \text{for}  \ (i,l)\in I^{\infty}.$$
Here, we understand $f_{i0}=1$ and $f_{il}=0$ for $l<0$.

\vskip 3mm

\begin{proposition}\cite{Bozec2014b,Bozec2014c} \
{\rm For any family $\nu=(\nu_{il})_{(i,l)\in I^{\infty}}$ of non-zero elements in $\Q(q)$, there exists a symmetric bilinear form $( \ , \ )_L :\mathscr F\times\mathscr F\rightarrow \Q(q)$ such that

\begin{itemize}
\item[(a)] $(x, y)_{L} =0$ if $|x| \neq |y|$,

\item[(b)] $(1,1)_{L} = 1$,

\item[(c)] $(f_{il}, f_{il})_{L} = \nu_{il}$ for all $(i,l) \in
I^{\infty}$,

\item[(d)] $(x, yz)_{L} = (\delta(x), y \otimes z)_L$  for all $x,y,z
\in {\mathscr F}$.
\end{itemize}
Here, $(x_1\otimes x_2,y_1\otimes y_2)_L=(x_1,y_1)_L(x_2,y_2)_L$ for any $x_1,x_2,y_1,y_2\in {\mathscr F}$.}
\end{proposition}

\vskip 3mm

From now on, we assume that
\begin{equation} \label{eq:assumption}
\nu_{il} \in 1+q\Z_{\geq0}[[q]]\ \ \text{for all} \ (i,l)\in I^{\infty}.
\end{equation}
Then, the bilinear form $( \ , \ )_L$ is non-degenerate on $\mathscr F(i)=\bigoplus _{l\geq 1} {\mathscr F}_{-l \alpha_i}$ for $i\in I^{im} \backslash I^{iso}$.

\vskip 3mm

Let $\widehat {U}$ be the associative algebra over $\Q(q)$ with $\mathbf 1$
generated by the elements $q^h$ $(h\in P^{\vee})$ and $e_{il},
f_{il}$ $((i,l) \in I^{\infty})$ with defining relations
\begin{equation} \label{eq:rels}
\begin{aligned}
& q^0=\mathbf 1,\quad q^hq^{h'}=q^{h+h'} \ \ \text{for} \ h,h' \in P^{\vee} \\
& q^h e_{jl}q^{-h} = q^{l\alpha_j(h)} e_{jl}, \ \ q^h f_{jl}q^{-h} = q^{-l\alpha_j(h)} f_{jl}\ \ \text{for} \ h \in P^{\vee}, (j,l)\in I^{\infty}, \\
& \sum_{k=0}^{1-la_{ij}}(-1)^k
{\begin{bmatrix} 1-la_{ij} \\ k \end{bmatrix}}_i{e_i}^{{1-la_{ij}}-k}e_{jl}e_i^k=0 \ \ \text{for} \ i\in
I^{\text{re}}, \, i \neq (j,l), \\
& \sum_{k=0}^{1-la_{ij}}(-1)^k
{\begin{bmatrix} 1-la_{ij} \\ k \end{bmatrix}}_i{f_i}^{{1-la_{ij}}-k}f_{jl}f_i^k=0 \ \ \text{for} \ i\in
I^{\text{re}}, \, i \neq (j,l), \\
& e_{ik}e_{jl}-e_{jl}e_{ik} = f_{ik}f_{jl}-f_{jl}f_{ik} =0 \ \ \text{for} \ a_{ij}=0.
\end{aligned}
\end{equation}
We extend the grading by setting $|q^h|=0$ and $|e_{il}|= l \alpha_{i}$.

\vskip 3mm

The algebra $\widehat{U}$ is endowed with the co-multiplication
$\Delta: \widehat{U} \rightarrow \widehat{U} \otimes \widehat{U}$
given by
\begin{equation} \label{eq:comult}
\begin{aligned}
& \Delta(q^h) = q^h \otimes q^h, \\
& \Delta(e_{il}) = \sum_{m+n=l} q_{(i)}^{mn}e_{im}\otimes K_{i}^{-m}e_{in}, \\
& \Delta(f_{il}) = \sum_{m+n=l} q_{(i)}^{-mn}f_{im}K_{i}^{n}\otimes f_{in} .
\end{aligned}
\end{equation}
where $K_i=q_{i}^{h_i}$$(i \in I)$.

\vskip 3mm

Let $\widehat{U}^{\leq0}$ be the subalgebra of $\widehat{U}$ generated by $f_{il}$ and $q^h$, for all $(i,l) \in I^{\infty}$ and $h\in P^{\vee}$, and $\widehat{U}^+$ be the subalgebra
generated by $e_{il}$ for all $(i,l) \in I^{\infty}$. In \cite{Bozec2014b}, Bozec showed that one can extended $( \ , \ )_L$ to a symmetric bilinear form $( \ , \ )_L$ on $\widehat{U}$ satisfying
\begin{equation}
\begin{aligned}
& (q^h,1)_L=1,\ (q^h,f_{il})_L=0, \\
& (q^h,K_j)_L=q^{-\alpha_j(h)},\\
& (x,y)_L=(\omega(x),\omega(y))_L \ \ \text{for all} \ x,y\in \widehat{U}^+,
\end{aligned}
\end{equation}
where $\omega:\widehat{U}\rightarrow\widehat{U}$  is the involution defined by
$$\omega(q^h)=q^{-h},\ \omega(e_{il})=f_{il},\ \omega(f_{il})=e_{il}\ \ \text{for}\ h \in P^{\vee},\ (i,l)\in I^{\infty}.$$

\vskip 3mm

For any $x\in \widehat{U}$, we shall use the Sweedler's notation, and write
$$\Delta(x)=\sum x_{(1)}\otimes x_{(2)}.$$

\vskip 3mm

Following the Drinfeld double process, we define $\tilde{U}$ as the quotient of $\widehat{U}$ by the relations
\begin{equation}\label{drinfeld}
\sum(a_{(1)},b_{(2)})_L\omega(b_{(1)})a_{(2)}=\sum(a_{(2)},b_{(1)})_La_{(1)}\omega(b_{(2)})\ \ \text{for all}\ a,b \in \widehat{U}^{\leq0}
\end{equation}

\vskip 3mm

\begin{definition}
Given a Borcherds-Cartan datum $(A, P, \Pi, P^{\vee}, \Pi^{\vee})$, the {\it quantum Borcherds-Bozec algebra} $U_q(\g)$ is defined to be the quotient algebra
of $\tilde{U}$ by the radical of $( \ , \ )_L$ restricted to $\tilde{U}^-\times\tilde{U}^+$.
\end{definition}

\vskip 3mm

Let $U^+$(resp. $U^-$) be the subalgebra of $U_q(\g)$ generated by $e_{il}$(resp. $f_{il}$) for all $(i,l)\in I^{\infty}$.
We will denote by $U^{0}$ the subalgebra of $U_q(\g)$ generated by $q^h$ for all $h\in P^{\vee}$. It is easy to see that $q^h$$(h\in P^{\vee})$ is a
$\Q(q)$-basis of $U^0$.

\vskip 3mm

In \cite{Kang2019b}, Kang and Kim showed that the co-multiplication $\Delta: \widehat{U} \rightarrow \widehat{U} \otimes \widehat{U}$ passes down to $U_q{(\g)}$ and with this, $U_q{(\g)}$ becomes a Hopf algebra. They also proved the quantum Borcherds-Bozec algebra has a {\it triangular decomposition}.

\vskip 3mm

\begin{theorem}\cite{Kang2019b}\label{tdc}
{\rm The the quantum Borcherds-Bozec algebra $U_q(\g)$ has the following triangular decomposition:
\begin{equation}\label{trian}
U_q{(\g)}\cong U^-\otimes U^0 \otimes U^+.
\end{equation}}
\end{theorem}

\vskip 3mm

By the defining relation \eqref{drinfeld}, we obtain complicated commutation relations between
$e_{il}$ and $f_{jk}$ for $(i, l), (j,k) \in I^{\infty}$. We will derive explicit formulas for these complicated
commutation relations in Appendix A. But, as we already see in \eqref{eq:un enveloping alg}, the commutation relations in the universal enveloping algebra $U(\g)$ of  Borcherds-Bozec algebra $\g$ are rather simple
\begin{equation}\label{eq:comrel}
e_{ik}f_{jl}-f_{jl}e_{ik} = k \, \delta_{ij} \, \delta_{kl} \, h_{i} \ \ \text{for} \ i,j \in I, k, l \in \Z_{>0}.
\end{equation}
Thanks to Bozec, there exists another set of generators in $U_q{(\g)}$ called {\it primitive generators}.
They satisfy a simpler set of commutation relations,
and we shall prove that these generators also satisfy all the defining relations of $U_q{(\g)}$ described in \eqref{eq:rels}.

\vskip 3mm

We denote by $\mathcal C_{l}$ (resp. $\mathcal P_{l}$) the set of compositions (resp. partitions) of $l$, and denote by $\eta:U_q(\g)\rightarrow U_q(\g)$ the $\Q$-algebra homomorphism defined by
\begin{equation}
\eta(e_{il})=e_{il},\ \eta(f_{il})=f_{il},\ \eta(q^h)=q^{-h},\ \eta(q)=q^{-1}\ \ \text{for} \ h\in P^{\vee},\ (i,l)\in I^{\infty}.
\end{equation}
As usual, let $S:U_q(\g)\rightarrow U_q(\g)$ and $\epsilon:U_q(\g)\rightarrow \Q(q)$ be the {\it antipode} and the {\it counit} of $U_q(\g)$, respectively. Then, we have the following proposition.

\vskip 3mm

\begin{proposition}\cite{Bozec2014b,Bozec2014c}\label{prim}
{\rm For any $i\in I^{\text im}$ and $l\geq 1$, there exist unique elements $t_{il}\in  U^-_{-l \alpha_{i}}$ and $s_{il}=\omega (t_{il})$ such that
\begin{itemize}\label{bozec}
\item[(1)] $\Q (q) \left<f_{il} \mid l\geq 1\right>=\Q (q) \left<t_{il} \mid l\geq 1\right>$ and $\Q (q) \left<e_{il} \mid l\geq 1\right>=\Q (q) \left<s_{il} \mid l\geq 1\right>$,
\item[(2)] $(t_{il},z)_L=0$ for all $z\in \Q (q) \left<f_{i1} ,\cdots,f_{il-1}\right>$,\\
$(s_{il},z)_L=0$ for all $z\in \Q (q) \left<e_{i1} ,\cdots,e_{il-1}\right>$
\item[(3)] $t_{il}-f_{il}\in \Q (q) \left<f_{ik} \mid k<l \right>$ and $s_{il}-e_{il}\in \Q (q) \left<e_{ik} \mid k<l \right>$,
\item[(4)] $\eta(t_{il})=t_{il},\ \eta(s_{il})=s_{il}$,
\item[(5)] $\delta(t_{il})=t_{il}\otimes 1+1\otimes t_{il}, \ \delta(s_{il})=s_{il}\otimes 1+ 1\otimes s_{il}$,
\item[(6)] $\Delta(t_{il})=t_{il}\otimes 1+K_i^l\otimes t_{il}, \ \Delta(s_{il})=s_{il}\otimes K_i^{-l}+ 1\otimes s_{il}$,
\item[(7)] $S(t_{il})=-K_i^{-l}t_{il},\ S(s_{il})=-s_{il}K_i^{l}$.
\end{itemize}}
\end{proposition}

\vskip 3mm

If we set $\tau_{il}=(t_{il},t_{il})_L=(s_{il},s_{il})_L$, we have the following commutation relations in $U_q(\g)$
\begin{equation}\label{news}
s_{il}t_{jk}-t_{jk}s_{il}=\delta_{ij}\delta_{lk}\tau_{il}(K_i^l-K_i^{-l}).
\end{equation}

\vskip 3mm

  Assume that $i\in I^{\text im}$ and let $\mathbf c =(c_1,\cdots,c_m)$ be an element in $\mathcal C_{l}$ or in $\mathcal P_{l}$. We set
  $$t_{i,\mathbf c}=\prod_{j=1}^{m}t_{ic_j} \ \text{and} \ s_{i,\mathbf c}=\prod_{j=1}^{m}s_{ic_j}.$$
  Notice that $\{ t_{i,\mathbf c} \mid \mathbf c \in \mathcal C_{l}\}$ is a basis of  ${\mathscr F}_{-l \alpha_{i}}$.
\vskip 4mm

  For $i\in I^{\text{iso}}$ and $\mathbf c, \mathbf c' \in \mathcal P_{l}$, if $\mathbf c\neq \mathbf c' $, then by induction, we have
$$(t_{i,\mathbf c},t_{i,\mathbf c'})_L=(s_{i,\mathbf c},s_{i,\mathbf c'})_L=0 .$$
 For $i\in I^{\text{im}}\backslash I^{\text{iso}}$ and $\mathbf c, \mathbf c' \in \mathcal C_{l}$, if the partitions obtained by rearranging $\mathbf c$ and $\mathbf c'$  are not equal, then we have
$$(t_{i,\mathbf c},t_{i,\mathbf c'})_L=(s_{i,\mathbf c},s_{i,\mathbf c'})_L=0 .$$

\vskip 4mm

For each $i\in I^{\text re}$, we also use the notation $t_{i1}$ and $s_{i1}$. Here we set
$$t_{i1}=f_{i1},\quad s_{i1}=e_{i1}.$$
Sometimes, we simply write $t_i$(resp. $s_i$) instead of $t_{i1}$(resp. $s_{i1}$) in this case. By mimicking Definition $1.2.13$ in \cite{Lusztig}, we have the following definition.

\vskip 3mm

\begin{definition}\label{dil}
For every $(i,l)\in I^{\infty}$, we define the linear maps $e_{i,l}',e_{i,l}'':\mathscr F\rightarrow\mathscr F$ by
\begin{equation}\label{delta}
e_{i,l}'(1)=0, \ e_{i,l}'(t_{jk})=\delta_{ij}\delta_{lk}\ \text{and}\ e_{i,l}'(xy)=e_{i,l}'(x)y+q^{l(|x|,\alpha_i)}xe_{i,l}'(y)
\end{equation}
\begin{equation}\label{delta1}
e_{i,l}''(1)=0, \ e_{i,l}''(t_{jk})=\delta_{ij}\delta_{lk}\ \text{and}\ e_{i,l}''(xy)=q^{l(|y|,\alpha_i)}e_{i,l}''(x)y+xe_{i,l}''(y)
\end{equation}
for any homogeneous elements $x,y$ in $\mathscr F$.
\end{definition}

\vskip 3mm

\begin{proposition}\label{radical}\
{\rm
\begin{enumerate}

\vskip 2mm

\item[(a)] For any $x,y\in \mathscr F$, we have
$$(t_{il}y,x)_L=\tau_{il}(y,e_{i,l}'(x))_L,\ (y t_{il},x)_L=\tau_{il}(y,e_{i,l}''(x))_L$$
\item[(b)] The maps $e_{i,l}'$ and $e_{i,l}''$ preserve the radical of $( \ , \ )_L$.
\item[(c)] Let $x\in U^-$, we have
\begin{itemize}
\item[(i)] If $e_{i,l}'(x)=0$ for all $(i,l)\in I^{\infty}$, then $x=0$.
\item[(ii)] If $e_{i,l}''(x)=0$ for all $(i,l)\in I^{\infty}$, then $x=0$.
\end{itemize}
\end{enumerate}}
\end{proposition}
\begin{proof} (a) For any homogeneous element $x\in \mathscr F$. We first show that
\begin{equation}\label{d1}
\delta(x)=t_{il}\otimes e_{i,l}'(x)+\sum_{w\neq (i,l)}t_w\otimes y_w,
\end{equation}
where if $w=(j_1,l_1)...(j_r, l_r)$ is a word in $I^{\infty}$, $t_w=t_{(j_1,l_1)}\cdots t_{(j_r, l_r)}$ and $y_w$ is an element in $\mathscr F$ depending on $w$.

\vskip 2mm

Since $e_{i,l}'$ is a linear map, it is enough to check \eqref{d1} by assuming that $x$ is a monomial in $t_{jk}$. Fix $(i,l)\in I^{\infty}$. We use induction on the number of $t_{il}$  that appears in $x$. If $x$ contains no $t_{il}$, then $e_{i,l}'(x)=0$ and there is no term of the form $t_{il}\otimes-$. Now assume that $x$ contains $t_{il}$, then we can write $x=x_1t_{il}x_2$ for some monomials $x_1,x_2$ such that $x_1$  doesn't contains $t_{il}$. So we have
\begin{equation}
e_{i,l}'(x)=e_{i,l}'(x_1t_{il}x_2)=q^{l(|x_1|,\alpha_i)}x_1e_{i,l}' {(t_{il}x_2)} =q^{l(|x_1|,\alpha_i)}x_1[x_2+q^{l(-l\alpha_i,\alpha_i)}t_{il}e_{i,l}'(x_2)].
\end{equation}
\vskip 2mm

\noindent On the other hand
\begin{equation}
\delta(x)=\delta(x_1)(t_{il}\otimes 1+1\otimes t_{il})\delta(x_2).
\end{equation}
By induction hypothesis, the term $t_{il}\otimes-$ only appear in
\begin{equation}
(1\otimes x_1)(t_{il}\otimes 1)(1 \otimes x_2)+(1\otimes x_1)(1\otimes t_{il})(t_{il} \otimes e_{i,l}'(x_2)),
\end{equation}
which is equal to
\begin{equation}
\begin{aligned}
&t_{il}\otimes q^{(|x_1|,l\alpha_i)}x_1x_2+t_{il}\otimes q^{-(|x_1|-l\alpha_i,-l\alpha_i)}x_1t_{il}e_{i,l}'(x_2) \\
&=t_{il}\otimes q^{l(|x_1|,\alpha_i)}x_1[x_2+q^{-l(l\alpha_i,\alpha_i)}t_{il}e_{i,l}'(x_2)].
\end{aligned}
\end{equation}
This shows \eqref{d1}.
\vskip 2mm

Similarly, we can show that
\begin{equation}\label{d2}
\delta(x)=e_{i,l}''(x)\otimes t_{il}+\sum_{w\neq (i,l)} z_w\otimes t_w.
\end{equation}
Since $e_{i,l}'$ and $e_{i,l}''$ are linear maps, the equations \eqref{d1} and \eqref{d2} hold for any $x,y\in \mathscr F$.
\vskip 2mm

For any $\mathbf c\in \mathcal C_{il} $, we have $(t_{il},t_{i\mathbf c})_L=\delta_{(l),\mathbf c}\tau_{il}$. Thus
\begin{equation}
(t_{il}y,x)_L=\tau_{il}(y,e_{i,l}'(x))_L,\ (y t_{il},x)_L=\tau_{il}(y,e_{i,l}''(x))_L
\end{equation}
for any $x,y\in \mathscr F$.
\vskip 3mm

(b) Since
$\tau_{il}=(t_{il},t_{il})_L\neq 0$,
our assertion follows.
\vskip 3mm
(c) Note that each monomial ends with some $t_{jk}$'s. By (a), if $e_{i,l}''(x)=0$ for all $(i,l)\in I^{\infty}$,  then $x$ belongs to the radial of $( \ , \ )_L$, which is equal to $0$ in $U^-$.
\end{proof}

\vskip 3mm
For any $i\in I^{\text{re}}$ and $n \in \N$, we set
$$t_i^{(n)}=\frac{t_i^{n}}{[n]_i!}.$$
By a similar argument as \cite[1.4.2]{Lusztig}, we have the following Lemma.

\vskip 3mm
\begin{lemma}
{\rm We have
\begin{equation}
\delta (t_i^{(n)})=\sum_{p+p'=n}q_i^{-pp'}t_i^{(p)}\otimes t_i^{(p')}
\end{equation}}
for any $i\in I^{\text{re}}$ and $n \in \N$.
\end{lemma}

\vskip 3mm

\begin{proposition}
{\rm For any $i\in I^{\text {re}}$, $(j,l)\in I^{\infty}$, and $i\neq (j,l)$, we have
\begin{equation*}
\sum_{p+p'=1-la_{ij}}(-1)^pt_i^{(p)}t_{jl}t_i^{(p')}=0
\end{equation*}
in $U_q{(\g)}$.}
\end{proposition}

\begin{proof}
If $i\in I^{\text{re}}$, we have $a_{ij}=\frac{2(\alpha_i,\alpha_j)}{(\alpha_i,\alpha_i)}$. Set
$$R_{i,(j,l)}=\sum_{p+p'=1-la_{ij}}(-1)^pt_i^{(p)}t_{jl}t_i^{(p')}.$$
By \eqref{radical}, we only need to show that $e''_\mu(R_{i,(j,l)})=0$ for all $\mu \in I^{\infty}$. It is clear that
$$e''_\mu(R_{i,(j,l)})=0 \  \ \text{if} \ \mu\neq i, (j,l).$$

\vskip 2mm
By the definition of $e''_i$, we have
\begin{equation}
\begin{aligned}
& e''_i(t_i^{(p)}t_{jl}t_i^{(p')})=q^{(\alpha_i,-p'\alpha_i)}e''_i(t_i^{(p)}t_{jl})t_i^{(p')}+t_i^{(p)}t_{jl}e''_i(t_i^{(p')}) \\
&=q^{-p'(\alpha_i,\alpha_i)}q^{-(\alpha_i,l\alpha_j)}q_i^{(1-p)}t_i^{(p-1)}t_{jl}t_i^{(p')}+q_i^{(1-p')}t_i^{(p)}t_{jl}t_i^{(p'-1)}.
\end{aligned}
\end{equation}
\vskip 2mm
Thus
\begin{equation}\label{sum}
\begin{aligned}
& e''_i(R_{i,(j,l)})=\sum_{p+p'=1-la_{ij}}(-1)^pq^{-p'(\alpha_i,\alpha_i)}q^{-(\alpha_i,l\alpha_j)}q_i^{(1-p)}t_i^{(p-1)}t_{jl}t_i^{(p')} \\
& \phantom{e''_i(R_{i,(j,l)})=} +\sum_{p+p'=1-la_{ij}}(-1)^pq_i^{(1-p')}t_i^{(p)}t_{jl}t_i^{(p'-1)} \\
&\phantom{e''_i(R_{i,(j,l)})}=\sum_{0\leq p \leq 1-la_{ij}}(-1)^pq^{-(1-la_{ij}-p)(\alpha_i,\alpha_i)}q^{-(\alpha_i,l\alpha_j)}q_i^{(1-p)}t_i^{(p-1)}t_{jl}t_i^{(1-la_{ij}-p)}\\
&\phantom{e''_i(R_{i,(j,l)})=} +\sum_{0\leq p \leq 1-la_{ij}}(-1)^pq_i^{(la_{ij}+p)}t_i^{(p)}t_{jl}t_i^{(-la_{ij}-p)}.
\end{aligned}
\end{equation}
The coefficient of $t_i^{(p)}t_{jl}t_i^{(-la_{ij}-p)}$ in the first sum of \eqref{sum} is
\begin{equation}
\begin{aligned}
&(-1)^{p+1}q^{-(-la_{ij}-p)(\alpha_i,\alpha_i)}q^{-(\alpha_i,l\alpha_j)}q_i^{-p} \\
&=(-1)^{p+1}q^{(l\frac{2(\alpha_i,\alpha_j)}{(\alpha_i,\alpha_i)}+p)(\alpha_i,\alpha_i)-l(\alpha_i,\alpha_j)+(-p)\frac{(\alpha_i,\alpha_i)}{2}}\\
&=(-1)^{p+1}q^{l(\alpha_i,\alpha_j)+p\frac{(\alpha_i,\alpha_i)}{2}}\\
&=(-1)^{p+1}q_i^{(la_{ij}+p)}.
\end{aligned}
\end{equation}
Hence, we have $e''_i(R_{i,(j,l)})=0$.
\vskip 2mm

By the definition of $e''_{jl}$, we have
\begin{equation}
e''_{jl}(t_i^{(p)}t_{jl}t_i^{(p')})=q^{-l{(\alpha_j,p'\alpha_i)}}e''_{jl}(t^{(p)}_it_{jl})t_i^{(p')}=q^{-l{(\alpha_j,p'\alpha_i)}}t_i^{(p)}t_i^{(p')}.
\end{equation}
So
\begin{equation}
e''_{jl}(R_{i,(j,l)})=\sum_{0\leq p'\leq1-la_{ij}}(-1)^{(1-la_{ij}-p')}q^{-l{(\alpha_j,p'\alpha_i)}}t_i^{(1-la_{ij}-p')}t_i^{(p')}.
\end{equation}
By \cite[1.3.4]{Lusztig} , we obtain
$$\sum_{0\leq p'\leq1-l\frac{2(\alpha_i,\alpha_j)}{(\alpha_i,\alpha_i)}}(-1)^{(1-l\frac{2(\alpha_i,\alpha_j)}{(\alpha_i,\alpha_i)}-p')}q^{-l{(\alpha_j,p'\alpha_i)}}\begin{bmatrix} 1-l\frac{2(\alpha_i,\alpha_j)}{(\alpha_i,\alpha_i)} \\ p' \end{bmatrix}_i=0.$$
Hence, we get $e''_{jl}(R_{i,(j,l)})=0$. This finishes the proof.
\end{proof}

\vskip 3mm

By the above arguments, we have primitive generators $t_{il}$$((i,l)\in I^{\infty})$ in $U^-$ of degree $-l\alpha_{i}$ and $s_{il}$$((i,l)\in I^{\infty})$ in $U^+$ of degree $l \alpha_{i}$ satisfying
\begin{equation}
\begin{aligned}\label{2.24}
& s_{il}t_{jk}-t_{jk}s_{il}=\delta_{ij}\delta_{lk}\tau_{il}(K_i^l-K_i^{-l}),\\
& \sum_{k=0}^{1-la_{ij}}(-1)^k
{\begin{bmatrix} 1-la_{ij} \\ k \end{bmatrix}}_i{t_i}^{{1-la_{ij}}-k}t_{jl}t_i^k=0 \ \ \text{for} \ i\in
I^{\text{re}}, \, i \neq (j,l).
\end{aligned}
\end{equation}
By using the involution $\omega$, we get
\begin{equation}
\sum_{k=0}^{1-la_{ij}}(-1)^k
{\begin{bmatrix} 1-la_{ij} \\ k \end{bmatrix}}_i{s_i}^{{1-la_{ij}}-k}s_{jl}s_i^k=0 \ \ \text{for} \ i\in
I^{\text{re}}, \, i \neq (j,l).
\end{equation}

\vskip 3mm

Since $t_{il}$(resp. $s_{il}$) can be written as a homogeneous polynomial of $f_{ik}$(resp. $e_{ik}$) for $k\leq l$, we have
  \begin{equation}\label{2.26}
q^h t_{jl}q^{-h} = q^{-l\alpha_j(h)} t_{jl}, \ \ q^h s_{jl}q^{-h} = q^{l\alpha_j(h)} s_{jl}\ \ \text{for} \ h \in P^{\vee}, (j,l)\in I^{\infty}, \\
\end{equation}
and
 \begin{equation}
[t_{ik},t_{jl}] = [s_{ik},s_{jl}] =0 \ \ \text{for} \ a_{ij}=0.
\end{equation}
 \vskip 8mm
\section{$\mathbb A_1$-form of the quantum Borcherds-Bozec algebras}

\vskip 2mm

We consider the localization of $\Q[q]$ at the ideal $(q-1)$:
\begin{equation}
\begin{aligned}
& \mathbb A_1=\{ f(q)\in \Q(q) \mid f \ \text{is regular at} \ q=1\} \\
&\ \quad =\{g/h \mid g,h \in  \Q[q],\ h(1)\neq 0  \}
\end{aligned}
\end{equation}
Let $\mathbb J_1$ be the unique maximal ideal of the local ring  $\mathbb A_1$, which is generated by $(q-1)$. Then we have an isomorphism of
fields
$${\mathbb A_1}/{\mathbb J_1} \xrightarrow {\sim} \Q , \quad f(q)+\mathbb J_1\mapsto f(1).$$

\vskip 3mm

Note that, for $i\in I^{\text re}$, $[n]_i$ and ${\begin{bmatrix} n \\ k \end{bmatrix}}_i$ are elements of $\Z[q,q^{-1}]\subseteq \mathbb A_1$. For any $h\in P^{\vee}$, $n \in \Z$, we formally define
$$(q^h;n)_q=\frac{q^hq^n-1}{q-1} \in U^0.$$

\begin{definition}
We define the $\mathbb A_1$-{\it form}, denote by $U_{\mathbb A_1}$ of the quantum Borcherds-Bozec algebra $U_q{(\g)}$ to be the $\mathbb A_1$-subalgebra generated by the elements $s_{il}$, $T_{il}$, $q^h$ and $(q^h;0)_q$, for all $(i,l)\in I^{\infty}$ and $\ h\in P^{\vee}$, where
\begin{equation}
T_{il}=\frac{1}{\tau _{il}}\frac{1}{q_i^2-1} t_{il} \ \ \text{for} \  (i,l)\in I^{\infty}.
\end{equation}
\end{definition}

\vskip 3mm

Let $U_{\mathbb A_1}^+$(resp. $U_{\mathbb A_1}^-$) be the $\mathbb A_1$-subalgebra of $U_{\mathbb A_1}$ generated by the elements $s_{il}$(resp. $T_{il}$) for $(i,l)\in I^{\infty}$, and $U_{\mathbb A_1}^0$ be the subalgebra of $U_{\mathbb A_1}$ generated by $q^h$ and $(q^h;0)_q$ for $(h\in P^{\vee})$.

\vskip 3mm
\begin{lemma}\label{lemma}\
{\rm \begin{enumerate}
\item[(a)] $(q^h;n)_q \in U_{\mathbb A_1}^0$ for all $n \in \Z$ and $h \in P^{\vee}$.
\item[(b)] $\dfrac{K_i^l-K_i^{-l}}{q_i^2-1} \in U_{\mathbb A_1}^0$.
\end{enumerate}}
\begin{proof}
It is straightforward to check that
\begin{equation}
\begin{aligned}
&(q^h;n)_q=q^n(q^h;0)_q+\frac{q^n-1}{q-1},\\
&\frac{K_i^l-K_i^{-l}}{q_i^2-1}=\frac{q-1}{q_i^2-1}(1+K_i^{-l})\frac{K_i^l-1}{q-1}.
\end{aligned}
\end{equation}
The lemma follows.
\end{proof}
\end{lemma}

\vskip 3mm

The next proposition shows that the triangular decomposition \eqref{trian} of $U_q{(\g)}$ carries over to its $\mathbb A_1$-form.

\vskip 3mm

\begin{proposition}\label{tr}
{\rm We have a natural isomorphism of $\mathbb A_1$-modules
\begin{equation}\label{triangular}
U_{\mathbb A_1}\cong U_{\mathbb A_1}^{-} \otimes U_{\mathbb A_1}^0 \otimes U_{\mathbb A_1}^+
\end{equation}
induced from the triangular decomposition of $U_q{(\g)}$.}
\begin{proof}
Consider the canonical isomorphism $\varphi: U_q{(\g)}\xrightarrow {\sim} U^-\otimes U^0 \otimes U^+$ given by multiplication.
By \eqref{2.24} and \eqref{2.26}, we have the following commutation relations
\begin{equation}
\begin{aligned}
& s_{il}(q^h;0)_q=(q^h;-l\alpha_i(h))_q s_{il},\\
&(q^h;0)_qT_{il}=T_{il}{(q^h;-l\alpha_i(h))}_q,\\
& s_{il}T_{jk}-T_{jk}s_{il}=\delta_{ij}\delta_{lk}\frac{K_i^l-K_i^{-l}}{q_i^2-1}.
\end{aligned}
\end{equation}
Combining with \eqref{lemma}, we can see that the image of $\varphi$ lies inside $U_{\mathbb A_1}^{-} \otimes U_{\mathbb A_1}^0 \otimes U_{\mathbb A_1}^+$.
\end{proof}
\end{proposition}

\vskip 3mm

The representation theory of quantum Borcherds-Bozec algebras has been studied by Kang  and
Kim in \cite{Kang2019b}. In the following sections, we will use some notions defined in \cite{Kang2019b}, which are similar to those in classical representation theory of quantum groups.

\vskip 3mm

Fix $\lambda \in P$, let $V^q$ be a highest weight $U_q{(\g)}$-module with highest weight $\lambda$ and highest weight vector $v_\lambda$. Then we have the $\mathbb A_1$-form for the highest weight modules.
\vskip 3mm

\begin{definition}
The $\mathbb A_1$-{\it form} of $V^q$ is defined to be the $U_{\mathbb A_1}$-module $V_{\mathbb A_1}=U_{\mathbb A_1}v_\lambda$.
\end{definition}

\vskip 3mm

By the definition of highest weight module and $V_{\mathbb A_1}$, it is easy to see that $V_{\mathbb A_1}=U_{\mathbb A_1}^-v_\lambda$. The highest weight
$U_q{(\g)}$-module $V^q$ has the weight space decomposition
\begin{equation}
V^q=\bigoplus_{\mu\leq\lambda}V_\mu^q,
\end{equation}
where $V_\mu^q=\{v \in V^q \mid q^hv=q^{\mu(h)}v \ \ \text{for all} \ h\in P^{\vee} \}$. For each $\mu \in P$, we define the {\it weight space} $(V_{\mathbb A_1})_\mu=V_{\mathbb A_1}\cap V_\mu^q$. The following proposition shows that $V_{\mathbb A_1}$ also has the weight space decomposition.

\vskip 3mm

\begin{proposition}
$V_{\mathbb A_1}=\bigoplus_{\mu\leq\lambda}(V_{\mathbb A_1})_\mu$
\begin{proof}
The proof is the same as \cite[Proposition 3.3.6]{HK2002}.
\end{proof}
\end{proposition}

\vskip 3mm

\begin{proposition}
{\rm For each $\mu \in P$, the weight space  $(V_{\mathbb A_1})_\mu$ is a free $\mathbb A_1$-module with $\text {rank}_{\mathbb A_1}(V_{\mathbb A_1})_\mu=\text { dim}_{\Q (q)}V_\mu^q$.}
\end{proposition}
\begin{proof}
We first show that $(V_{\mathbb A_1})_\mu$ is finite generated as an $\mathbb A_1$-module. Since we have $V_{\mathbb A_1}=U_{\mathbb A_1}^-v_\lambda$, every element in $V_{\mathbb A_1}$ is a polynomial of $T_{il}$ with coefficients in $\mathbb A_1$. Assume that $\lambda=\mu+\alpha$ for some $\alpha \in Q_+$. Then for each $v \in \mathbb A_1$ with weight $\mu$, $v$ must be a $\mathbb A_1$-linear combination of $\{ T_{i_1l_1}\cdots T_{i_pl_p}v_\lambda \mid l_1\alpha_{l_1}+\cdots l_p\alpha_{l_p}=\alpha \}$, which is a finite set.
\vskip 2mm

Let $\{ T_\zeta v_\lambda \}$ be a $\Q(q)$-basis of $V_\mu^q$, where $T_\zeta$ are monomials in $T_{il}$. The set $\{ T_\zeta v_\lambda \}$ certainly belongs to $(V_{\mathbb A_1})_\mu$ and is also $\mathbb A_1$-linearly independent. So we have $\text{rank}_{\mathbb A_1}(V_{\mathbb A_1})_\mu\geq\text{dim}_{\Q(q)}V_\mu^q$. Let
$\{u_1,\cdots,u_p\}$ be an $\mathbb A_1$-linearly independent subset of $(V_{\mathbb A_1})_\mu$. Consider a $\Q(q)$-linear dependence relation
$$c_1(q)u_1+\cdots +c_p(q)u_p=0,\ c_k(q)\in \Q(q) \ \ \text{for} \ k=1,\cdots,p.$$
Multiplying some powers of $(q-1)$ if needed, we may assume that all $c_k(q)\in \mathbb A_1$, which implies that $c_k(q)=0 \ \text{for all} \ k=1,\cdots,p$. Hence $u_1,\cdots,u_p$ are linearly independent over $\Q(q)$ and $\text{rank}_{\mathbb A_1}(V_{\mathbb A_1})_\mu\leq\text{dim}_{\Q(q)}V_\mu^q$, which completes the proof.
\end{proof}

\begin{corollary}
{\rm The $\Q(q)$-linear map $\varphi: \Q(q)\otimes_{\mathbb A_1}V_{\mathbb A_1}\rightarrow V^q$ given by $c\otimes v\mapsto cv$ is an isomorphism.}
\end{corollary}

\section{Classical limit of quantum Borcherds-Bozec algebras}

\vskip 2mm
 Define the $\Q$-linear vector spaces
 \begin{equation}
 \begin{aligned}
 & U_1=({\mathbb A_1}/{\mathbb J_1})\otimes_{\mathbb A_1}U_{\mathbb A_1}\cong {U_{\mathbb A_1}}/{\mathbb J_1U_{\mathbb A_1}}, \\
 & V^1=({\mathbb A_1}/{\mathbb J_1})\otimes_{\mathbb A_1}V_{\mathbb A_1}\cong {V_{\mathbb A_1}}/{\mathbb J_1V_{\mathbb A_1}}.
 \end{aligned}
 \end{equation}
 Then $V^1$ is naturally a $U^1$-module. Consider the natural maps
\begin{equation}
 \begin{aligned}
 & U_{\mathbb A_1}\rightarrow U_1={U_{\mathbb A_1}}/{\mathbb J_1U_{\mathbb A_1}}, \\
&V_{\mathbb A_1}\rightarrow V^1={V_{\mathbb A_1}}/{\mathbb J_1V_{\mathbb A_1}}.
\end{aligned}
 \end{equation}
The passage under these maps is referred to as taking the classical limit. We will denote by $\overline{x}$ the image of $x$ under the classical limit. Notice that $q$ is mapped to $1$ under these maps.
\vskip 3mm

For each $\mu \in P$, set $V_\mu^1=({\mathbb A_1}/{\mathbb J_1})\otimes_{\mathbb A_1}(V_{\mathbb A_1})_\mu$. Then we have

\vskip 3mm

\begin{proposition}\label{dim}\
{\rm \begin{enumerate}
\item [(a)] $V^1=\bigoplus_{\mu\leq \lambda}V_\mu^1$.
\item [(b)] For each $\mu \in P$, $\text{dim}_{\Q}V_{\mu}^1=\text{rank}_{\mathbb A_1}(V_{\mathbb A_1})_\mu=\text{dim}_{\Q(q)}V_\mu^q$.
\end{enumerate}}
\end{proposition}

\vskip 3mm

Let $\overline{h} \in U_1$ denote the classical limit of the element $(q^h;0)_q\in U_{\mathbb A_1}$. As in \cite{HK2002}, we have the following lemma.

\vskip 3mm

\begin{lemma}\
{\rm \begin{enumerate}
\item [(i)] For all $h \in P^{\vee}$, we have $\overline{q^h}=1$.
\item [(ii)] For any $h,h' \in P^{\vee},\ \overline{h+h'}=\overline{h}+\overline{h'}$. Hence, we have $\overline{nh}=n\overline{h}$ for $n\in \Z$.
\end{enumerate}}
\end{lemma}

\vskip 3mm

 Define the subalgebras $U_1^0=\Q \otimes U_{\mathbb A_1}^0$ and $U_1^\pm=\Q \otimes U_{\mathbb A_1}^\pm$.
 The next theorem shows that we can define a
 surjective homomorphism from the universal enveloping algebra $U(\g)$ to $U_1$,
 and as a $U(\g)$-module,
 $V^1$ is a highest weight module with highest weight $\lambda \in P$ and highest weight vector $\overline{v}_\lambda$.

\vskip 3mm

\begin{theorem}\label{sur}\
{\rm \begin{enumerate}
\item [(a)] The elements $\overline{s}_{il},\ \overline{T}_{il}$$((i,l) \in I^{\infty})$ and $\overline{h}$$(h \in P^{\vee})$ satisfy the defining relations of $U(g)$. Hence there exists a surjective $\Q$-algebra homomorphism $\psi: U(\g)\rightarrow U_1$ sending $e_{il}$ to $\overline{s}_{il}$, $f_{il}$ to $\overline{T}_{il}$, $h$ to
     $\overline{h}$. In particular, the $U_1$-module $V^1$ has a $U{(\g)}$-module structure.
\item [(b)] For each $\mu \in P$, $h \in P^{\vee}$, the element $\overline{h}$ acts on $V_\mu^1$ as scalar multiplication by $\mu(h)$. So $V_\mu^1$ is the $\mu$-weight space of the $U(\g)$-module $V^1$.
\item [(c)] As a $U(\g)$-module, $V^1$ is a highest weight module with highest weight $\lambda \in P$ and highest weight vector $\overline{v}_\lambda$.
\end{enumerate}}
\begin{proof}
(a) Since $\dfrac{K_i^l-K_i^{-l}}{q_i^2-1}=\dfrac{q-1}{q_i^2-1}(1+K_i^{-l})\dfrac{K_i^l-1}{q-1}$, when we take classical limit, we get
$$\overline{\dfrac{K_i^l-K_i^{-l}}{q_i^2-1}}=\dfrac{1}{2r_i}\cdot2\cdot lr_i\overline{h}_i=l\overline{h}_i.$$
By \eqref{2.24}, we have the following equation in $U_1$ $$\overline{s}_{il}\overline{T}_{jk}-\overline{T}_{jk}\overline{s}_{il}=\delta_{ij}\delta_{lk}l\overline{h}_i,$$
and it is the same as the commutation relations in $U(\g)$.
\vskip 2mm
Since we have
$$q^h s_{jl} = q^{l\alpha_j(h)} s_{jl}q^{h}, \ \ q^h T_{jl} = q^{-l\alpha_j(h)} T_{jl}q^{h} \ \ \text{for} \ h \in P^{\vee}, (j,l)\in I^{\infty},$$
we get $\dfrac{q^h-1}{q-1}s_{il}=s_{il}\dfrac{q^{l\alpha_i(h)}q^h-1}{q-1}$ and
\begin{equation}
 \dfrac{q^h-1}{q-1}s_{il}-s_{il}\dfrac{q^h-1}{q-1}=s_{il}\dfrac{q^{l\alpha_i(h)}-1}{q-1}q^h.
\end{equation}
 Thus $\overline{h}\overline{s}_{il}-\overline{s}_{il}\overline{h}=l\alpha_i(h)\overline{s}_{il}$.
 Similarly, we have $$\overline{h} \ \overline{T}_{il}-\overline{T}_{il} \overline{h}=-l\alpha_i(h)\overline{T}_{il}.$$
\vskip 2mm
It is easy to check the commutation relations
\begin{equation}
[\overline{T}_{ik},\overline{T}_{jl}] = [\overline{s}_{ik},\overline{s}_{jl}] =0 \ \ \text{for} \ a_{ij}=0.
\end{equation}

\vskip 2mm
For $i\in I^{\text{re}}$, we have
$$\overline{[n]}_i=n \ \ \text{and}\ \overline{{\begin{bmatrix} n \\ k \end{bmatrix}}}_i=\binom n k.$$
Hence the remaining Serre relations follow.
\vskip 2mm
(b) For $v \in (V_{\mathbb A_1})_\mu$ and $h \in P^{\vee}$, we have $(q^h;0)_qv=\dfrac{q^{\mu(h)-1}}{q-1}v$. Hence when we take the classical limit, we obtain $\overline{h}v=\mu(h)v$.
\vskip 2mm
(c) As a $U(\g)$-module, by (2), we have $h\overline{v}_\lambda=\overline{h}\overline{v}_\lambda=\lambda(h)\overline{v}_\lambda$ in $V^1$ for all $h \in P^{\vee}$. For each $(i,l)\in I^{\infty}$, $s_{il}\overline{v}_\lambda$ is zero. Therefore, $V^1=U_1^-\overline{v}_\lambda=U^-(\g)\overline{v}_\lambda$ and hence $V^1$ is a highest weight module with highest weight $\lambda \in P$ and highest weight vector $\overline{v}_\lambda$.
\end{proof}
\end{theorem}

\vskip 3mm

Combining Proposition \ref{dim} (b) and Theorem \ref{sur} (b), we have $\text{ch}V^1=\text{ch}V^q$. For a dominant integral weight $\lambda \in P^+$, the irreducible highest weight $U_q(\g)$-module $V^q(\lambda)$ has the following property.

\vskip 3mm
\begin{proposition}\label{irr highest}\cite{Kang2019b}
{\rm Let $\lambda \in P^+$ and $V^q(\lambda)$ be the irreducible highest weight module with highest weight $\lambda$ and highest weight vector $v_\lambda$. Then the following statements hold.
\begin{itemize}
\item [(a)] If $i \in I^{\text{re}}$, then $f_i^{\lambda(h_i)+1}v_\lambda=0$.
\item [(b)] If $i \in I^{\text{im}}$ and $\lambda(h_i)=0$, then $f_{ik}v_\lambda=0$ for all $k>0$.
\end{itemize}}
\end{proposition}

\vskip 3mm

We now conclude that the classical limit of the
irreducible highest weight $U_{q}(\g)$-module $V^q(\lambda)$ is isomorphic to the irreducible
highest $U(\g)$-module $V(\lambda)$.

\begin{theorem}
{\rm If $\lambda \in P^+$ and $V^q$ is the irreducible highest weight $U_q(\g)$-module $V^q(\lambda)$ with highest weight $\lambda$, then $V^1$ is isomorphic to the irreducible highest weight module $V(\lambda)$ over $U(\g)$ with highest weight $\lambda$.}
\begin{proof}
By Proposition \ref{irr highest}, if $i\in I^{\text{re}}$, then $T_i^{\lambda(h_i)+1}v_\lambda=0$; if $i \in I^{\text{im}}$ and $\lambda(h_i)=0$, then $T_{ik}v_\lambda=0$ for all $k>0$. Therefore, $V^1$ is a highest weight $U_q(\g)$-module with highest weight $\lambda$ and highest weight vector $\overline{v}_\lambda$ satisfying:
\begin{itemize}
\item [(a)] If $i \in I^{\text{re}}$, then $f_i^{\lambda(h_i)+1}\overline{v}_\lambda=\overline{T}_i^{\lambda(h_i)+1}\overline{v}_\lambda=0$.
\item [(b)] If $i \in I^{\text{im}}$ and $\lambda(h_i)=0$, then $f_{ik}\overline{v}_\lambda=\overline{T}_{ik}\overline{v}_\lambda=0$ for all $k>0$.
\end{itemize}
Hence $V^1\cong V(\lambda)$ by Proposition \ref{prop:hwmodule}.
\end{proof}
\end{theorem}

\vskip 3mm
By Proposition \ref{dim} (b), the character of $V^q(\lambda)$ is the same as the character of $V(\lambda)$, which is given by (see, \cite{Kang2019,BSV2016})
\begin{equation}
\begin{aligned}
\text{ch} V(\lambda) &= \dfrac{\sum_{w \in W} \epsilon (w) e^{w(\lambda +
\rho)-\rho} w (S_{\lambda})}{\prod_{\alpha \in \Delta_{+}} (1-
e^{-\alpha})^{\dim \g_{\alpha}}} \\
&= \dfrac{\sum_{w \in W} \sum_{s \in F_{\lambda}} \epsilon(w)
\epsilon(s) e^{w(\lambda + \rho -s) - \rho}}{\prod_{\alpha \in
\Delta_{+}} (1-e^{-\alpha})^{\dim \g_{\alpha}}}.
\end{aligned}
\end{equation}
\vskip 3mm

\begin{theorem}
{\rm The classical limit $U_1$ of $U_q(\g)$ is isomorphic to the universal enveloping algebra $U(\g)$ as $\Q$-algebras. }
\begin{proof}
By Theorem \ref{sur} (a), we already have an epimorphism $\psi:U(\g)\twoheadrightarrow U_1$ sending $e_{il}$ to $\overline{s}_{il}$, $f_{il}$ to $\overline{T}_{il}$, $h$ to $\overline{h}$, respectively. So it is sufficient to show that $\psi$ is injective.
\vskip 2mm
We first show that the restriction $\psi_0$ of $\psi$ to $U^0(\g)$ is an isomorphism of $U^0(\g)$ onto $U_1^0$. Note that $\psi_0$ is certainly surjective. Since $\chi=\{h_i\mid i\in I\}\cup\{d_i\mid i\in I\}$ is a $\Z$-basis of the free $\Z$-lattice $P^{\vee}$, it is also a $\Q$-basis of the Cartan subalgebra $\h$. Thus any element of $U^0(\g)$ may be written as a polynomial in $\chi$. Suppose $g \in \text{Ker}\psi_0$. Then, for each $\lambda \in P$, we have
$$0=\psi_0(g)\cdot \overline{v}_\lambda=\lambda(g)\overline{v}_\lambda,$$
where $v_\lambda$ is a highest weight vector of a highest weight $U_q(\g)$-module of highest weight $\lambda$ and $\lambda(g)$ denotes the polynomial in $\{\lambda(x)\mid x \in \chi \}$ corresponding to $g$. Hence, we have $\lambda(g)=0$ for every $\lambda \in P$. Since we may take any integer value for $\lambda(x)$$(x\in \chi)$, $g$ must be zero, which implies that $\psi_0$ is injective.
\vskip 2mm
Next, we show that the restriction of $\psi$ to $U^-(\g)$, denote by $\psi_-$, is an isomorphism of $U^-(\g)$ onto $U_1^-$. Suppose $\text{Ker}\psi_-\neq0$, and take a  non-zero element $u=\sum a_\zeta f_\zeta \in\text{Ker}\psi_-$, where $a_\zeta \in \Q$ and $f_\zeta$ are monomials in $f_{il}$'s $(i,l)\in I^{\infty}$. Let $N$ be the maximal length of the monomials $f_\zeta$ in the expression of $u$ and choose a dominant integral weight $\lambda \in P^+$ such that $\lambda(h_i)>N$ for all $i \in I$. The kernel of the $U^-(\g)$-module homomorphism $\varphi:U^-(\g)\rightarrow V^1$ given by $x\mapsto \psi(x)\cdot \overline{v}_\lambda$ is the left ideal of $U^-(\g)$ generated by $f_i^{\lambda(h_i)+1}$$( i\in I^{\text{re}})$ and $f_{il}$ for $i\in I^{\text{im}}$ with  $\lambda(h_i)=0$. Because of the choice of $\lambda$, it is generated by $f_i^{\lambda(h_i)+1}$ for all $i\in I^{\text{re}}$.
\vskip 2mm
 Therefore, $u=\sum a_\zeta f_\zeta \notin \text{Ker} \varphi$. That is, $\psi_-(u)\cdot\overline{v}_\lambda=\psi(u)\cdot\overline{v}_\lambda\neq 0$, which is a contradiction. Therefore, $\text{Ker}\psi_-=0$ and $U^-(\g)$ is isomorphic to $U_1^-$.
 \vskip 2mm
 Similarly, we have $U^+(\g)\cong U_1^+$. Hence, by the triangular decomposition, we have the linear isomorphisms
 $$U(\g)\cong U^-(\g)\otimes U^0(\g) \otimes U^+(\g)\cong U_1^-\otimes U_1^0 \otimes U_1^+\cong U_1,$$
 where the last isomorphism follows from Proposition \ref{tr}. It is easy to see that this isomorphism is actually an algebra isomorphism.
\end{proof}
\end{theorem}

 \vskip 3mm
 We now show that $U_1$ inherits a Hopf algebra structure from that of $U_q(\g)$. It suffices to show that $U_{\mathbb A_1}$ inherits the Hopf algebra structure from that of $U_q(\g)$. Since we have
\begin{equation}\label{hopf}
\begin{aligned}
&\Delta(T_{il})=T_{il}\otimes 1+K_i^l\otimes T_{il}, \ \Delta(s_{il})=s_{il}\otimes K_i^{-l}+ 1\otimes s_{il},\\
&\Delta(q^h)=q^h\otimes q^h,\\
&S(T_{il})=-K_i^{-l}T_{il},\ S(s_{il})=-s_{il}K_i^{l},\ S(q^h)=q^{-h},\\
&\epsilon(T_{il})=\epsilon(s_{il})=0,\ \epsilon(q^h)=1,
\end{aligned}
\end{equation}
we get
 \begin{equation}\label{hopf1}
\begin{aligned}
&\Delta((q^h;0)_q)=\frac{q^h\otimes q^h-1\otimes1}{q-1}=(q^h;0)_q\otimes1+q^h\otimes (q^h;0)_q,\\
& S((q^h;0)_q)=(q^{-h};0)_q,\\
&\epsilon((q^h;0)_q)=0.
\end{aligned}
\end{equation}
Hence the maps $\Delta\colon U_{\mathbb A_1}\rightarrow U_{\mathbb A_1}\otimes U_{\mathbb A_1}$, $\epsilon\colon U_{\mathbb A_1}\rightarrow \mathbb A_1$, and
$S\colon U_{\mathbb A_1}\rightarrow U_{\mathbb A_1}$ are all well-defined and $U_{\mathbb A_1}$ inherits a Hopf algebra structure from that of $U_q(\g)$.

\vskip 3mm

Let us show that the Hopf algebra structure of $U_1$ coincides with that of $U(\g)$ under the isomorphism we have been considering. Taking the classical limit of the equations in \eqref{hopf} and in \eqref{hopf1}, we have
\begin{equation}
\begin{aligned}
&\Delta(\overline{T}_{il})=\overline{T}_{il}\otimes 1+1\otimes \overline{T}_{il}, \ \Delta(\overline{s}_{il})=\overline{s}_{il}\otimes 1+ 1\otimes \overline{s}_{il},\ \Delta(\overline{h})=\overline{h}\otimes 1+1\otimes \overline{h},\\
&S(\overline{T}_{il})=-\overline{T}_{il},\ S(\overline{s}_{il})=-\overline{s}_{il},\ S(\overline{h})=-\overline{h},\\
&\epsilon(\overline{T}_{il})=\epsilon(\overline{s}_{il})=\epsilon(\overline{h})=0.
\end{aligned}
\end{equation}
This coincides with \eqref{str}. Therefore, we have the following corollary.
\vskip 3mm

\begin{corollary}
{\rm The classical limit $U_1$ of $U_q(\g)$ inherits a Hopf algebra structure from that of $U_q(\g)$ so that $U_1$  and $U(\g)$ are isomorphic as Hopf algebras over $\Q$.
}
\end{corollary}
 \vskip 3mm
 Since $U^-(\g)\cong U_1^-$, by the same argument in \cite[Theorem 3.4.10]{HK2002}, we have the following theorem when we take the classical limit on the Verma module over $U_q(\g)$.
\vskip 3mm

\begin{theorem} \cite{HK2002}
{\rm If $\lambda \in P$ and $V^q$ is the Verma module $M^q(\lambda)$ over $U_q(\g)$ with highest weight $\lambda$, then its classical limit $V^1$ is isomorphic to the Verma module $M(\lambda)$ over $U(\g)$ with highest weight $\lambda$.}
\end{theorem}

\clearpage
\appendix

\section{}
\vskip 2mm

We shall provide an explicit commutation relations for $e_{ik}$ and $f_{jl}$, for $(i,k),(j,l)\in I^{\infty}$ in $U_q(\g)$. Recall that, we have the co-multiplication formulas
$$\Delta(f_{il}) = \sum_{m+n=l} q_{(i)}^{-mn}f_{im}K_{i}^{n}\otimes f_{in}.$$
Then, the defining relation \eqref{drinfeld} yields the following lemma.
\vskip 3mm

\begin{lemma}\cite{Kang2019b}
{\rm For any $i,j\in I$ and $k,l\in\Z_{>0}$, we have
 \begin{itemize}
\item [(a)] If $i\neq j$, then $e_{ik}$ and $f_{jl}$ are commutative.
\item [(b)] If $i=j$, we have the following relations in $U_q(\g)$ for all $k,l>0$
\begin{equation}\label{ij}
\sum_{\substack{m+n=k \\ n+s=l}}q_{(i)}^{n(m-s)}\nu_{in}e_{is}f_{im}K_i^{-n}=\sum_{\substack{m+n=k \\ n+s=l}}q_{(i)}^{-n(m-s)}\nu_{in}f_{im}e_{is}K_i^{n}.
\end{equation}
 \end{itemize}}
\end{lemma}

\vskip 3mm
Since we have
$$K_i^ne_{im}K_i^{-n}=q_{(i)}^{2nm}e_{im},$$
$$K_i^nf_{im}K_i^{-n}=q_{(i)}^{-2nm}f_{im}.$$
We can modify the equations \eqref{ij} as the following form
\begin{equation}\label{mn}
\sum_{\substack{m+n=k \\ n+s=l}}q_{(i)}^{n(s-m)}\nu_{in}K_i^{-n}e_{is}f_{im}=\sum_{\substack{m+n=k \\ n+s=l}}q_{(i)}^{n(m-s)}\nu_{in}K_i^{n}f_{im}e_{is}.
\end{equation}
\vskip 3mm

If $i\in I^{\text{re}}$, then $k=l=1$ and $m=s$, so there are only one commutation relation in this case
\begin{equation}
e_if_i+\nu_{i1}K_i^{-1}=f_ie_i+\nu_{i1}K_i.
\end{equation}
\vskip 3mm

If $i\in I^{\text{im}}$ (we omit the notation $``i"$ in this case for simplicity), we first assume that $k=l$. By \eqref{mn}, we have
\begin{equation}
\begin{aligned}
& k=l=1,\quad e_1f_1+\nu_{1}K^{-1}=f_1e_1+\nu_{1}K,\\
& k=l=2,\quad e_2f_2+\nu_{1}K^{-1}e_1f_1+\nu_{2}K^{-2}=f_2e_2+\nu_{1}Kf_1e_1+\nu_2K^2,\\
&\qquad \qquad  \qquad \qquad  \ \ \ \ \cdots \\
& k=l=n,\quad e_nf_n+\nu_{1}K^{-1}e_{n-1}f_{n-1}+\cdots+\nu_{n-1}K^{1-n}e_1f_1+\nu_nK^{-n}\\
&\phantom{k=l=n,\quad} =f_n e_n+\nu_{1}K f_{n-1}e_{n-1}+\cdots+\nu_{n-1}K^{n-1}f_1 e_1+\nu_nK^{n}.\\
\end{aligned}
\end{equation}
\vskip 3mm
By direct calculation, we can write $e_nf_n-f_ne_n$ in the following way
$$e_nf_n-f_ne_n=\alpha_1f_{n-1} e_{n-1}+\alpha_2f_{n-2} e_{n-2}+\cdots+\alpha_{n-1}f_1e_{1}+\alpha_n,$$
where
\begin{equation}
\begin{aligned}
& \alpha_1=\nu_1(K-K^{-1}),\\
& \alpha_2=\nu_2(K^2-K^{-2})-\nu_1K^{-1}\alpha_1=\nu_2(K^2-K^{-2})-\nu_1^2K^{-1}(K-K^{-1}),\\
& \alpha_3=\nu_3(K^3-K^{-3})-\nu_1K^{-1}\alpha_2-\nu_2K^{-2}\alpha_1\\
& \phantom{\alpha_3}=\nu_3(K^3-K^{-3})-\nu_1\nu_2K^{-1}(K^2-K^{-2})+(\nu_1^3-\nu_1\nu_2)K^{-2}(K-K^{-1}),\\
& \qquad \qquad  \qquad \qquad  \ \ \ \ \cdots \\
& \alpha_n=\nu_n(K^n-K^{-n})-\nu_1K^{-1}\alpha_{n-1}-\nu_2K^{-2}\alpha_{n-2}-\cdots-\nu_{n-1}K^{-(n-1)}\alpha_1.
\end{aligned}
\end{equation}

\vskip 3mm

If $m\in \N$ and $\mathbf c=(c_1,\cdots,c_d)$ is a composition of $m$ (i.e. $\mathbf c\in \mathcal C_m$), then we set
$\nu_{\mathbf c}=\prod_{k=1}^{d} \nu_k$ and $\lVert \mathbf c \rVert=d$.

\vskip 3mm
By induction, we have
\begin{equation}\label{k=l}
e_nf_n=\sum_{p=1}^{n}\left\{\sum_{r=1}^{p}\left[\nu_r\vartheta_{p-r}K^{r-p}(K^r-K^{-r})\right]\right\}f_{n-p}e_{n-p}+f_ne_n,
\end{equation}
where $\vartheta_{m}=\sum_{\mathbf c \in \mathcal C_m}(-1)^{\lVert \mathbf c \rVert}\nu_{\mathbf c}$. For example, $\vartheta_{4}=\nu_1^4-3\nu_1^2\nu_2+2\nu_1\nu_3+\nu_2^2-\nu_4$.
\vskip 3mm

Next, we assume that $k-l=t$, then $m-s=t$. By \eqref{mn}, we get
$$\sum_{n=0}^{l}q_{(i)}^{-nt}\nu_{n}K^{-n}e_{l-n}f_{k-n}=\sum_{n=0}^{l}q_{(i)}^{nt}\nu_{n}K^{n}f_{k-n}e_{l-n}.$$
Hence, we have
\begin{equation*}
\begin{aligned}
& e_lf_k+q_{(i)}^{-t}\nu_1K^{-1}e_{l-1}f_{k-1}+\cdots+q_{(i)}^{-(l-1)t}\nu_{l-1}K^{-(l-1)}e_{1}f_{t+1}
+q_{(i)}^{-lt}\nu_{l}K^{-l}f_{t}\\
&= f_k e_l+q_{(i)}^{t}\nu_1Kf_{k-1}e_{l-1}+\cdots+q_{(i)}^{(l-1)t}\nu_{l-1}K^{(l-1)}f_{t+1}e_{1}
+q_{(i)}^{lt}\nu_{l}K^{l}f_{t}.
\end{aligned}
\end{equation*}
We substitute $K$ by $q_{(i)}^tK$ in formula \eqref{k=l} and obtain
\begin{equation}\label{k-l=t}
e_lf_k=\sum_{p=1}^{l}\left\{\sum_{r=1}^{p}\left[\nu_r\vartheta_{p-r}(q_{(i)}^tK)^{r-p}((q_{(i)}^tK)^r-(q_{(i)}^tK)^{-r})\right]\right\}f_{k-p}e_{l-p}+f_ke_l.
\end{equation}

\vskip 3mm

Finally, we assume that $l-k=t$, then $s-m=t$. By \eqref{mn}, we get
$$\sum_{n=0}^{k}q_{(i)}^{nt}\nu_{n}K^{-n}e_{l-n}f_{k-n}=\sum_{n=0}^{k}q_{(i)}^{-nt}\nu_{n}K^{n}f_{k-n}e_{l-n}.$$
Hence, we have
\begin{equation*}
\begin{aligned}
& e_lf_k+q_{(i)}^{t}\nu_1K^{-1}e_{l-1}f_{k-1}+\cdots+q_{(i)}^{(l-1)t}\nu_{l-1}K^{-(l-1)}e_{t+1}f_{1}
+q_{(i)}^{lt}\nu_{l}K^{-l}e_{t}\\
&= f_k e_l+q_{(i)}^{-t}\nu_1Kf_{k-1}e_{l-1}+\cdots+q_{(i)}^{-(l-1)t}\nu_{l-1}K^{(l-1)}f_{1}e_{t+1}
+q_{(i)}^{-lt}\nu_{l}K^{l}e_{t}.
\end{aligned}
\end{equation*}
We substitute $K$ by $q_{(i)}^{-t}K$ in formula \eqref{k=l} and obtain
\begin{equation}\label{l-k=t}
e_lf_k=\sum_{p=1}^{k}\left\{\sum_{r=1}^{p}\left[\nu_r\vartheta_{p-r}(q_{(i)}^{-t}K)^{r-p}((q_{(i)}^{-t}K)^r-(q_{(i)}^{-t}K)^{-r})\right]\right\}f_{k-p}e_{l-p}+f_ke_l.
\end{equation}

\vskip 3mm
Combine the formulas \eqref{k=l}, \eqref{k-l=t}, and \eqref{l-k=t}, we have the following statement.

\vskip 3mm

\begin{proposition}
{\rm For $i \in I^{\text{im}}$, we have the following commutation relations for all $k,l>0$
\begin{equation}
e_{il}f_{ik}-f_{ik}e_{il}=\sum_{p=1}^{\text{min}\{k,l\}}\left\{\sum_{r=1}^{p}
\left[\nu_{ir}\vartheta_{i,p-r}(q_{(i)}^{k-l}K_i)^{r-p}((q_{(i)}^{k-l}K_i)^r-(q_{(i)}^{k-l}K_i)^{-r})\right]\right\}f_{i,k-p}e_{i,l-p}.
\end{equation}
Where $\vartheta_{i,p-r}=\sum_{\mathbf c \in \mathcal C_{p-r}}(-1)^{\lVert \mathbf c \rVert}\nu_{i\mathbf c}$}.
\end{proposition}

\end{document}